\documentclass[11pt]{amsart} 
\usepackage{xcolor}

\usepackage{amsmath}
\usepackage{amscd} 
\usepackage{amssymb} 
\usepackage{ascmac}
\usepackage{color}
\usepackage{ulem}
\usepackage{fixmath}
\usepackage{mathrsfs}
\usepackage{mathtools}
\usepackage{multicol}
\usepackage{enumitem}
\usepackage{marginfix}
\usepackage{pifont}
\usepackage{soul}
\usepackage{tcolorbox}

\usepackage[margin=1.0in]{geometry} 
 
\DeclareMathOperator{\CP}{CP}

\DeclareMathOperator{\dist}{dist}

\DeclareMathOperator{\diam}{diam}

\DeclareMathOperator{\spt}{spt}
\DeclareMathOperator{\Rat}{Rat}
\newcommand{\C}{\mathbb{C}}
\newcommand{\N}{\mathbb{N}}

\newcommand{\R}{\mathbb{R}}

\theoremstyle{definition}
\newtheorem{thm}{Theorem}[section]
\newtheorem{dfn}[thm]{Definition}
\newtheorem{prp}[thm]{Proposition}
\newtheorem{lem}[thm]{Lemma}
\newtheorem{cor}[thm]{Corollary}

\newtheorem*{notation*}{Notation}
\newtheorem*{claim*}{Claim}
\newtheorem*{mra}{Main result A}
\newtheorem*{mrb}{Main result B}
\newtheorem*{mrc}{Main result C}
\newtheorem*{mrd}{Main result D}
\newtheorem*{mre}{Main result E}

\title{Bowen's formula for a rational graph-directed Markov system}

\author{Tadashi Arimitsu, Johannes Jaerisch, Hiroki Sumi and Takayuki Watanabe}

\address{(Tadashi Arimitsu) Graduate School of Mathematics, Nagoya University,
Furocho, Chikusaku, Nagoya, 464-8602, JAPAN} 
\email{m21002u@math.nagoya-u.ac.jp}

\address{(Johannes Jaerisch) Graduate School of Mathematics, Nagoya University,
Furocho, Chikusaku, Nagoya, 464-8602, JAPAN} 
\email{jaerisch@math.nagoya-u.ac.jp}

\address{(Hiroki Sumi) Graduate School of Human and Environmental Studies, Kyoto University, Yoshida Nihonmatsu-cho, Sakyo-ku, Kyoto, 606-8501, Japan}
\email{sumi@math.h.kyoto-u.ac.jp}

\address{(Takayuki Watanabe) Chubu University Academy of Emerging Sciences/Center for Mathematical Science and Artificial Intelligence. 1200 Matsumotocho, Kasugai City, Aichi Prefecture 487-8501, Japan}
\email{takawatanabe@isc.chubu.ac.jp}

\subjclass[2020]{37F10, 58F23, 30D05, 37H12}

\begin{document} 

\begin{abstract}
We establish Bowen's formula for the Julia set of a non-elementary, expanding, irreducible and aperiodic rational graph-directed Markov system satisfying the backward separating condition. 
Towards this end, we shall prove that the associated skew product map is topologically exact on the skew product Julia set, and satisfies the density of repelling periodic points.
 Moreover, we give a criterion for expandingness in terms of hyperbolicity. 
\end{abstract}
\maketitle

\section{Introduction and statement of results}
Let Rat be the set of all non-constant rational maps on the Riemann sphere $\widehat{\C}$. 
A subsemigroup of $\Rat$ with semigroup operation being the functional composition is called a rational semigroup. 
The study of rational semigroups was initiated independently in  \cite{96gongren}, \cite{MR1397693},\cite{MR1625944} for their own purposes.
This theory has seen significant development in relation to independent and identically-distributed (i.i.d.) random holomorphic dynamics \cite{MR2747724, MR3084426, MR4268827}. 
Here, by i.i.d. we mean that the holomorphic maps are chosen independently and with the same distribution. We refer to \cite{MR2747724} and \cite{MR3084426} for precise definitions.

Regarding further development along this line of study, it is noticeable that the authors of \cite{MR4002398} and \cite{MR4407234} developed this random iteration theory by introducing a Markov random dynamical system of rational maps. 
These works revealed new phenomena which cannot be observed in i.i.d. random dynamical systems; see, for example,  \cite[Proposition 4.23]{MR4002398}. 
We aim to establish a fractal theory for Markov random dynamical systems of rational maps which will lead us to a new horizon of random complex dynamics.
 As a first step, we establish Bowen's formula for the Hausdorff dimension of the Julia set of a rational graph-directed Markov system (rational GDMS for short). 
We refer to Section \ref{WAaUimgpMB} for an introduction to rational GDMS. 

To explain and motivate our study of Bowen's formula better, let us first briefly introduce the thermodynamic formalism of dynamical systems. 
The origin of this theory dates back to the seminal paper \cite{MR0399421} by Y. Sina\u{\i}. Later, D. Ruelle and R. Bowen laid the foundation of this theory as cofounders; see \cite{MR2129258}, \cite{MR2423393} for comprehensive accounts of the prototypical results. 
Of particular concern is Bowen's celebrated result which determines the Hausdorff dimension of a quasi-circle as the unique zero of a central notion in the thermodynamic formalism called pressure function; see \cite{MR0556580}. 
This achievement is widely known as Bowen's formula.
In \cite{MR0684247}, Ruelle further developed this result in the context of hyperbolic Julia sets, and many other analogous statements have been justified in various settings to this date.
In particular, we note that under the open set condition, Bowen's formula for a finitely generated, expanding rational semigroup was justified in \cite{MR2153926}.

\subsection{Rational graph directed Markov systems}\label{WAaUimgpMB}

Let us begin with a formal introduction to a rational GDMS. 
All the definitions in this subsection are given in \cite{MR4002398} and  \cite{MR648108}.
  \begin{dfn} 
Let $(V, E)$ be a \textit{directed graph} with finite \textit{vertex set} $V$ and finite \textit{edge set} $E$.
Let $\Gamma_e \subset \Rat$ be a non-empty, finite subset of $\Rat$ indexed by a directed edge $e \in E$.
Then the triple $S :=(V, E,(\Gamma_e)_{e \in E})$ is called a \textit{rational GDMS}. 
The symbol $i(e)$ (resp. $t(e))$ denotes the \textit{initial} (resp. \textit{terminal}) \textit{vertex} of each directed edge $e \in E$.
\end{dfn}
   We shall adopt the convention that $\N:=\{n \in \mathbb{Z} :  n>0\}$ and  $\N_{0}:=\{n \in \mathbb{Z} :  n \geq 0\}$.
For each point of the complex plane $\C$, let $|z|$ denote the modulus of $z$.

 \begin{dfn}
Given a rational GDMS $S := (V, E,(\Gamma_e)_{e \in E})$, we define  the subshift of finite type $(X(S), \sigma)$ given by 
$$
X(S) := \left\{\xi := (\xi_{n})_{n \in \N}:= (g_{n}, e_{n})_{n \in \N} \in(\Rat \times E)^{\N} \mid \forall n \in \N, g_n \in \Gamma_{e_n}, t(e_n)=i(e_{n+1})\right\},
$$ 
and the left shift $\sigma: X(S) \rightarrow X(S)$.    
We say that $S$ is \textit{irreducible} if $(X(S),\sigma)$ is irreducible, and  $S$ is \textit{irreducible and aperiodic} if $(X(S),\sigma)$ is irreducible and aperiodic; see e.g. \cite{MR648108} for these properties of subshifts. 
    
For each  $N \in \N$ we denote by $X^{N}(S)$ the set of all subwords of $X(S)$ of length $N$.
For each element $\xi = (g_n, e_n)_{1\le n \le N} \in X^{N}(S)$, denoting $g_\xi := g_N\circ \dots \circ g_1$, we define the sets
\begin{align*}
&H(S) := \bigcup_{N \in \N}\left\{g_{\xi} \mid \xi \in X^{N}(S) \right\}, \quad   H_{i}(S) := \left\{g_{\xi} \mid \xi \in X^{N}(S), i=i(e_{1})\right\},\\
&H^{j}(S) := \bigcup_{N \in \N}\left\{g_{\xi}  \mid \xi \in X^{N}(S), j=t(e_{N})\right\} \text{ and } H^{j}_{i}(S) := H_{i}(S) \cap H^{j}(S).
\end{align*}
The set $F(S)$ defined as $\{z \in  \widehat{\C} \mid H(S) \text { is equicontinuous on some neighborhood of } z\}$ is called the \textit{Fatou set of $S$}.
Its complement $J(S):= \widehat{\C}\setminus F(S)$ is called the \textit{Julia set of $S$}. 

Likewise, we define the following sets at each vertex.
The set $F_i(S)$ defined as $\{z \in  \widehat{\C} \mid H_i(S) \text { is equicontinuous on some neighborhood of } z\}$ is called the \textit{Fatou set of $S$ at the vertex $i$}.
Its complement $J_i(S):=\widehat{\C} \setminus F_i(S)$ is called the \textit{Julia set of $S$ at the vertex $i$}.
\end{dfn}

The cardinality of a set $\mathcal{A}$ is denoted by $\# \mathcal{A} \in \N_{0} \cup\{\infty\}$.
 \begin{dfn}
A rational GDMS $S$ is said to be \textit{non-elementary} if $\# J_{i}(S) \geq 3$ for all $i \in V$.
 \end{dfn}
\begin{dfn} Let $S := (V, E,(\Gamma_e)_{e \in E})$ be a rational GDMS.
\begin{enumerate}[label=(\roman*)]
\item We define the \textit{skew product map} $\tilde{f}: X(S) \times  \widehat{\C} \to X(S) \times  \widehat{\C}$ associated with $S$ by 
$$
\tilde{f}(\xi, z) :=\left(\sigma(\xi), g_1(z)\right)
$$
for each $\xi  := \left(g_n, e_n\right)_{n \in \N} \in X(S)$ and $z \in  \widehat{\C}$. 

\item For each $\xi := (g_n, e_n)_{n\in \N} \in X(S)$, we define the \textit{fiber Fatou set $F_{\xi}$ of $\xi$} by $$
\{z \in  \widehat{\C} \mid  (g_N\circ \dots \circ g_1)_{N \in \N} \text { is equicontinuous on some neighborhood of } z\}.
$$
Its complement $J_{\xi}:=\widehat{\C} \setminus F_{\xi}$ is called  the \textit{fiber Julia set of $\xi$}. 
We also define $F^{\xi}:=\{\xi\} \times F_{\xi} \subset X(S) \times  \widehat{\C}$, and $J^{\xi}:=\{\xi\} \times J_{\xi} \subset X(S) \times  \widehat{\C}$.
We call the set  
$$
J(\tilde{f}):=\overline{\bigcup_{\xi \in X(S)} J^{\xi}} \subset X(S) \times  \widehat{\C}
$$ the \textit{skew product Julia set of $\tilde{f}$}.
\end{enumerate}
\end{dfn}

 Let  $d_{\widehat{\C}}$ be the spherical distance on $\widehat{\C}$.
Let $d_{X}$ be the shift distance on $X(S)$ and define the distance $\tilde{d}$ on the set  $X(S) \times \widehat{\C}$ given by $\tilde{d}((\xi, x),(\tau, y)):= d_{X}(\xi, \tau)+d_{\widehat{\C}}(x, y)$ for each $(\xi, x),(\tau, y) \in X(S) \times \widehat{\C}$.
We shall endow the set $J(\tilde{f})$ with the subspace topology of the metric topology on $X(S) \times \widehat{\C}$.

\subsection{Main results}
In the setting above, we introduce further definitions to state the main results.
Let $S$ be a rational GDMS. 
Given an element $\xi \in X(S)$, we also write $(g_{\xi,n}, e_{\xi,n})_{n \in \N} := \xi$.
For each $\tilde{z} := (\xi, z) \in X(S) \times \widehat{\C}$ and $n \in \N$, we set 
$$
(\tilde{f}^n)^{\prime}(\tilde{z}) : =\left(g_{\xi,n} \circ \cdots \circ g_{\xi,1}\right)^{\prime}(z).
$$

The first result provides the characterisation of the skew product Julia set $J(\tilde{f})$ from the perspective of the dynamics of $\tilde{f}$ under a mild assumption.
 Recall that a point $\tilde{z} \in \widehat X(S) \times \C$ is called \textit{a repelling periodic point of $\tilde{f}$ } if there exists $p \in \N$ such that $\tilde{f}^{p}(\tilde{z}) = \tilde{z}$ and $|\tilde{f}^{p}(\tilde{z})^{\prime}| >1$.

\begin{mra}[see Proposition \ref{density_repelling_per_points}] 
Let $S$ be a non-elementary, irreducible rational GDMS and $\tilde{f}$ be the associated skew product map.
Then, $J(\tilde{f})$ is equal to the closure of 
the set of repelling periodic points of $\tilde{f}$.
\end{mra}
The next result concerns the topological behaviour of the skew product map $\tilde{f}$ with respect to $J(\tilde{f})$: the topological exactness, often casually described as locally eventually onto. 
To be more precise, the skew product map $\tilde{f}: X(S)\times \widehat{\C}\to X(S)\times \widehat{\C}$ is said to be \textit{topologically exact on $J(\tilde{f})$} if for each open subset $U$ of $X(S)\times \widehat{\C}$ with $U\cap J(\tilde{f})\neq \emptyset$, there exists an element $N\in \N$ such that for each $n\in \N$ with $n\geq N$, we have $\tilde{f}^{n}(U)\supset J(\tilde{f}).$

To obtain this result, we restrict the class of a skew product map $\tilde{f}$ by using the notion of expandingness. 
This notion is inspired by a class called the expanding rational semigroup introduced in \cite{MR2153926}.
We state the definition. A rational GDMS $S$ is said to be \textit{expanding} if $J(S) \neq \emptyset$ and the map $\tilde{f}$ is  \textit{expanding along fibers}, i.e.,  there exist $C > 0$ and $\lambda>1$ such that for each $n \in \N$
$$
\inf _{\tilde{z} \in J(\tilde{f})}\|(\tilde{f}^n)^{\prime}(\tilde{z})\| \geq C \lambda^n,
$$
where we use $\|\cdot\|$ to denote the norm of the derivative with respect to the metric $d_{\widehat{\C}}$.

\begin{mrb}[see Corollary \ref{ghraCYWJEm}]
Let $S$ be a non-elementary irreducible and aperiodic rational GDMS. 
Suppose that the associated skew product map $\tilde{f}$ is expanding along fibers.
Then, $\tilde{f}$ is topologically exact on $J(\tilde{f})$. 
\end{mrb}
Motivated by \cite{sumi1998hausdorff},  we introduce the definitions below for a rational GDMS. Let $S := \left(V, E,\left(\Gamma_e\right)_{e \in E}\right)$ be a rational GDMS.  For each vertex $i \in V$, the set defined by
$$
P_{i}(S) := \overline{\bigcup_{h \in H^{i}(S)}\{\text { critical values of } h\}} \subset \widehat{\C}
$$
is called the \textit{post-critical set at the vertex $i$}.
A rational GDMS $S$ is said to be \textit{hyperbolic} if $P_{i}(S) \subset F_{i}(S)$ for each $i \in V$.

\begin{mrc}[see Theorem \ref{A7J66dnGYg}]\label{jSk7Eo33Nc}
Let $S$ be a hyperbolic, non-elementary, irreducible rational GDMS.
 For each $i \in V$ and $h \in H_{i}^{i}(S)$, assume that $h$ is loxodromic when its degree is one.
 Then, the skew product map associated with $S$ is expanding along fibers.
\end{mrc}

To state the last result, we briefly introduce some relevant definitions of thermodynamic formalism; see subsection \ref{qdXaxWWPnV} for precise definitions.
Assume that a rational GDMS $S$ is expanding.
The geometric potential $\tilde\varphi \colon J(\tilde{f}) \to \R$ is given by $\tilde\varphi(\tilde{z}):= -\log \|g_{\xi, 1}^{\prime}(z)\|$ for each $\tilde{z} := (\xi, z) \in J(\tilde{f})$.
Let $P(\tilde{f}, \tilde{\varphi})$ denote the topological pressure of the pair $(\tilde{f}, \tilde\varphi)$.  
Then, there is a unique zero $\delta$ of the geometric pressure function $\mathcal{P}\colon \R \to \R$,  $t \mapsto P(\tilde{f}, t\tilde\varphi)$.
We refer to subsection \ref{qdXaxWWPnV} for further details on topological pressure.
Let $\dim_{H}(A)$ denote the Hausdorff dimension of a set $A \subset \widehat{\C}$ with respect to $d_{\widehat{\C}}$.

\begin{mrd}[see Theorem \ref{cI0oM7KYMj}]
Let $S$ be a non-elementary, expanding, irreducible and aperiodic rational GDMS.
Then, the $\delta$-dimensional Hausdorff measure of $J(S)$ is finite and its upper box-counting dimension never exceeds $\delta$.
\end{mrd}

\begin{mre}[see Corollary \ref{pYqFyMRHfD}]
Let $S$ be a non-elementary, expanding, irreducible and aperiodic  rational GDMS. 
Suppose that $S$ satisfies the backward separating condition.
Then, we have $\dim_{H}(J(S)) = \delta$ and the $\delta$-dimensional Hausdorff measure of $J(S)$ is positive and finite.
\end{mre}
The definition of the backward-separation condition is given in Definition \ref{wVNh1F6ao0}. 
In section \ref{8dEsBgjjIq}, we shall estimate the Hausdorff dimension of the Julia set for a rational GDMS by using the last result.

\section{Proof of main results}
For each $z \in \widehat{\C}$ and $r > 0$, let $D(z,r) := \{\zeta \in \widehat{\C} : d_{\widehat{\C}}(z, \zeta) < r\}$. 
 Let $\dist_{\widehat{\C}}(A, B):= \inf\{d_{\widehat{\C}}(z, w): z \in A, w \in B \}$ for each nonempty subsets $A,B \subset \widehat{\C}$. 
Let $\diam_{\widehat\C}(A)$ denote the diameter of the set $A$ with respect to $d_{\widehat{\C}}$.
For each element $g$ of $\Rat$, we write $\deg(g) \in \N$ to mean its degree.
In the present paper, we further use the following notation for a rational GDMS $S$.
\begin{notation*}
Let $N \in \N$ and $i \in V$.
\begin{itemize}
\item Let $X_{i}(S) :=  \left\{(g_{n}, e_{n})_{n \in \N} \in X(S) \mid i = i(e_{1})\right\}$.
\item For each $(\tau_{1}, \ldots, \tau_{N}) \in X^{N}(S)$, a \textit{cylinder set} of length $N$ is a set $[\tau_{1}, \ldots, \tau_{N}] := \{(\xi_{n})_{n \in \N}  \in X(S) \mid  1 \leq \forall n \leq N, \xi_{n} = \tau_{n}\}$.
\item Given $\xi \in X^{N}(S)$ and $\xi^{\prime} \in  X(S)$, we write $\xi\xi^{\prime} := (\xi, \xi^{\prime})$ if $(\xi, \xi^{\prime}) \in  X(S)$.
\item Denote the first and second projections as $\pi_{1} \colon X(S) \times \widehat{\C} \to X(S)$ and $\pi_{2} \colon X(S)\times \widehat{\C} \to  \widehat{\C}$ respectively.
\end{itemize}
\end{notation*}

Next, we list several important properties of a rational GDMS.
\begin{prp} \label{UHyK6FOUuf}
Let $S$ be a rational GDMS and $\tilde{f}$ be the skew product map associated with $S$. 
Then, we have the following.
\begin{enumerate}[label=(\roman*)]
\item\label{sc15xYijnn} $\tilde{f}^{-1}(J(\tilde{f})) = J(\tilde{f})$ and $\tilde{f}(J(\tilde{f})) = J(\tilde{f})$.
\item \label{XfUMBmTKD2} If the system $S$ is irreducible, then $J_i(S)=J\left(H_i^i(S)\right)$ for each $i \in V$.
\item The map $\tilde{f}$ is continuous on $X(S) \times \widehat\C$ and finite-to-one.
\item \label{homeo of f tilde} A point $(\xi, z) \in X(S) \times \widehat{\C}$  satisfies $g_{1}^{\prime}(z) \neq 0$  if and only if $\tilde{f}$ is a homeomorphism in any small neighborhood of $(\xi, z)$.
\item \label{projection}
$J(S) = \pi_{2}(\tilde J(\tilde{f}))$.
\end{enumerate}
\end{prp}
\begin{proof}
See Lemma 2.31 in  \cite{MR4002398} for the proof of \ref{sc15xYijnn}. 
The other assertions follow from definitions and are analogous to some statements given in section 3 of \cite{MR1767945}.
\end{proof}

\subsection{The density of repelling periodic points and the topological exactness of a skew product map}\label{0jGAbQXBRW}
The objectives of this subsection are to prove the density in $J(\tilde{f})$ of the repelling periodic points of a skew product map $\tilde{f}$ and, as a consequence, the topological exactness of $\tilde{f}$ under some standard assumptions.

\begin{prp}\label{density_repelling_per_points}
Let $S$ be a non-elementary, irreducible rational GDMS and $\tilde{f}: X(S)\times \widehat{\C}\to X(S)\times \widehat{\C}$ the associated skew product map.
Then, $J(\tilde{f})$ is equal to the closure of 
the set of repelling periodic points of $\tilde{f}. $ 
\end{prp}

\begin{proof}
Let $U$ be a nonempty open subset of $X(S)\times \widehat{\C}$ 
with $U\cap J(\tilde{f})\neq \emptyset.$ Then, there exists an element
 $(\xi, z_{0})\in U\cap J(\tilde{f}).$ We may assume 
 $z_{0}\in J_{\xi}.$ 
 Let $i:=i(e_{\xi, 1}). $ There exist $r\in \N $ and 
 an open subset $V_{0}$ of $\widehat{\C}$ such that 
 $$(\xi ,z_{0})\in [\xi _{1},\ldots, \xi _{r}]\times V_{0}\subset U.$$ 

  Since $S$ is irreducible and $V$ is finite, there exists a finite set $\mathscr{F} \subset H(S)$ such that for each $i,j\in V$ there exists $h\in H^i_j(S) \cap \mathscr{F}$.  
 For each $n \in \N$ we take  $h_{n}\in H_{t(e_{\xi,n})}^{i}(S)\cap \mathscr{F}$. Clearly, $\{ h_{n}\mid n\in \N\} $ is a finite set. 
 Since $z_{0}\in J_{\xi }$, the sequence 
 $\{ g_{\xi, n}\circ \cdots \circ g_{\xi ,1}\}_{n=1}^{\infty }$ is not 
 equicontinuous in any neighborhood of $z_{0}.$ 
 We now show the following claim.
\begin{claim*}\label{NSwtUL14So}
The sequence $\{ h_{n}\circ g_{\xi, n}\circ \cdots \circ g_{\xi ,1}\}_{n=1}^{\infty} $ is not equicontinuous in any neighborhood of $z_{0}$. 
\end{claim*}
 
 To prove this claim, 
 suppose $\{ h_{n}\circ g_{\xi, n}\circ \cdots \circ g_{\xi ,1}\}_{n=1}^{\infty }$ is  
 equicontinuous in a neighborhood of $z_{0}.$ Let 
 $\varepsilon >0$ be any small positive number. 
 Since $\{ h_{n}\mid n\in \N\}$ is finite, 
  there exists an $\varepsilon_{0}>0$ such that 
  for each $\varepsilon_{0}$-ball $B$ in $\widehat{\C}$ and for each 
  $n\in \N$, the diameter of any connected component of $h_{n}^{-1}(B)$ is less than $\varepsilon.$ For this $\varepsilon_{0}>0$, since 
  we are assuming that $\{ h_{n}\circ g_{\xi, n}\circ \cdots \circ g_{\xi ,1}\}_{n=1}^{\infty }$ is  
 equicontinuous in a neighborhood of $z_{0}$, there exists a 
 $\delta >0$ such that 
 for each $n\in \N$, the diameter of $h_{n}\circ g_{\xi, n}\circ \cdots \circ g_{\xi ,1}(D(z_{0}, \delta ))$ is less than $\varepsilon_{0}$. It follows that 
 the diameter of $g_{\xi, n}\circ \cdots \circ g_{\xi ,1}(D(z_{0}, \delta ))$ is 
 less than $\varepsilon.$ 
 However, this contradicts $z_{0}\in J_{\xi}.$ 
 Thus we have proved our claim. 
 
 Since $S$ is irreducible, $\# J_{i}(S)=\# J(H_{i}^{i}(S))\geq 3$ by (\ref{XfUMBmTKD2}) of Proposition \ref{UHyK6FOUuf}. 
 Hence, $\# J(H_{i}^{i}(S))= \infty$ and $J(H_{i}^{i}(S))$ is equal to the closure of the set of repelling fixed 
 points of elements of $H_{i}^{i}(S)$; see e.g. \cite{MR1397693}, \cite{MR2900562}.
 
  Thus there exist  mutually distinct three 
 elements $x_{1}, x_{2}, x_{3}\in J(H_{i}^{i})$ such that for each 
 $j=1,2,3,$ the 
point $x_{j}$ is a repelling fixed point of some $\alpha_{j}\in H_{i}^{i}(S)$.   
  By the claim above and Montel's theorem, it follows that 
there exist an element $n\in \N$ with $n\geq r$, an element 
$z_{1}\in V_{0}$, and an element $j\in \{ 1,2,3\}$ such that 
$$h_{n}\circ g_{\xi, n}\circ \cdots \circ g_{\xi ,1}(z_{1})=x_{j}.$$

Therefore there exist an element $\rho \in [\xi_{1}, \ldots, \xi _{r}]$ and 
an element $s\in \N$ with $s\geq r$ such that 
\begin{equation}\label{XRexk38dyn}
\ z_{1}\in J_{\rho },\ g_{\rho, s}\circ \cdots \circ g_{\rho ,1}\in H_{i}^{i}(S), 
\mbox{ and } g_{\rho, s}\circ \cdots \circ g_{\rho ,1}(z_{1})=x_{j}\in J_{i}(S)=J(H_{i}^{i}(S)).
\end{equation}
Let $H:=\{ h\in H_{i}^{i}(S)\mid h=\cdots \circ g_{\rho, s}\circ \cdots \circ g_{\rho ,1}\}$.
Note that this is a rational semigroup and 
\begin{equation}\label{tfIdUGgdiI}
\ J(H)=(g_{\rho , s}\circ \cdots \circ g_{\rho ,1})^{-1}(J(H_{i}^{i}(S)).
\end{equation}
By $(\ref{XRexk38dyn}),(\ref{tfIdUGgdiI})$, we obtain 
\begin{equation}\label{1gongIbvUM}
\ z_{1}\in (g_{\rho, s}\circ \cdots \circ g_{\rho ,1})^{-1}(J(H_{i}^{i}(S)))
=J(H).
\end{equation}
Moreover, since $\# J(H_{i}^{i}(S))\geq 3$, by $(\ref{tfIdUGgdiI})$ we have 
$\# (J(H))\geq 3.$ Hence,  
 $J(H)$ is equal to the closure of the set of all repelling fixed points 
of elements of $H$; see  \cite{MR1397693}. 
Combining this with $(\ref{1gongIbvUM})$ and the fact $z_{1}\in V_{0}$, we obtain that 
there exist an element $z_{2}\in V_{0}$ and an element $g\in H$ such that 
$z_{2}$ is a repelling fixed point of $g.$

  Hence, there exist an element $\beta \in [\rho _{1}, \ldots ,\rho _{s}]
  \subset [\xi _{1}, \ldots ,\xi _{r}]$ and an element $k\in \N$ with 
  $k\geq s$ such that $\sigma ^{k}(\beta )=\beta $ and $z_{2}$ is a 
  repelling fixed point of $g_{\beta, k}\circ \cdots g_{\beta ,1}.$ 
  Thus 
  $$(\beta ,z_{2})\in 
  ([\xi _{1},\ldots ,\xi _{r}]\times V_{0})\cap J(\tilde{f})\subset U\cap J(\tilde{f})$$ and 
  $(\beta, z_{2})$ is a repelling periodic point of $\tilde{f}.$ 
  Hence, we have proved our theorem. 
\end{proof}

\begin{dfn}\cite[Definition 4.5]{MR4002398}
A point $z$ is called \textit{an exceptional point} of $S$ at the vertex $i \in V$ if $\#\bigcup_{h \in H_i^i(S)} h^{-1}(\{z\})<3$, and let $\mathcal{E}_i(S)$ denote the set of all exceptional points $z$ of $S$ at $i \in V$.
\end{dfn}

\begin{lem}\label{ftilde_nk_u_naught}
Let $S:=(V, E,(\Gamma_e)_{e \in E})$ be a non-elementary, irreducible and aperiodic rational GDMS.
Suppose that $K$ is a compact subset of $X(S)\times \widehat{\C}$ such that $K\cap (X_{i}(S)\times \mathcal{E}_{i}(S))=\emptyset $ for each $i\in V$.
Furthermore, assume that for any open set $U \subset X(S) \times \widehat{\C}$ with $U \cap J(\tilde f ) \neq \emptyset$, there exists  $k \in \N$ and non-empty open set $U_{0} \subset U$ such that 
\begin{equation}\label{ftilde_k_U_naught_1}
U_0 \cap J(\tilde{f}) \neq \emptyset \quad \text{and} \quad \tilde{f}^k\left(U_0\right) \supset U_0.
\end{equation}
Then, there exists $N \in \N$ such that for all $n \geq N$, $\tilde{f}^{nk}\left(U_{0}\right) \supset K$.
\end{lem}

\begin{proof}
Fix $U$ in the assumption.
Then, Theorem \ref{density_repelling_per_points} allows us to take an arbitrary repelling $\tilde{f}$-periodic point $(\xi, z_0) \in U_0 \cap J(\tilde{f})$ of period $k \in \N$.
Here, there exist $r^{\prime} \in \N$ and a neighborhood $V_{0}$ of $z_{0}$ such that $r := r^{\prime}k$ and 
\begin{equation}\label{xi_r_v_naught_u_naught}
 U_{0} \supset [\xi_{1}, \cdots, \xi_{r}] \times V_{0}.
\end{equation}
Setting $V_{1} := (g_{\xi, r} \circ \cdots \circ g_{\xi, 1})(V_{0})$, we have
\begin{equation}\label{f_tilde_r_v_1}
\tilde{f}^r\left(\left[\xi_{1}, \cdots, \xi_r\right] \times V_0\right) \supset [\xi_{1}] \times V_{1}.
\end{equation}
It follows from (\ref{xi_r_v_naught_u_naught}) and (\ref{f_tilde_r_v_1}) that
$$
\bigcup_{n \in \N} \tilde{f}^{nk}\left(U_{0}\right)
\supset
\bigcup_{n \in \N} \tilde{f}^{nk}\left(\tilde{f}^r\left(\left[\xi_{1}, \cdots, \xi_r\right] \times V_0\right)\right)
\supset
\bigcup_{n \in \N} \tilde{f}^{nk}\left([\xi_{1}] \times V_1\right).
$$
Thus, it suffices to prove $\bigcup_{n \in \N} \tilde{f}^{nk}\left([\xi_{1}] \times V_1\right) \supset K$ to complete the proof.

Let $i\in V$ and take   and $\left(\rho, z_1\right) \in K \cap\left(\pi_1^{-1} X_i(S) \right)$. 
Then,  we have $z_1 \notin  \mathcal{E}_{i}(S)$ by the assumption. 
Under the setting of this argument, we verify the next assertion.
\begin{claim*}
$\# \pi_2(\bigcup_{n \in \N} \tilde{f}^{-nk}(\rho, z_1) \cap \pi_1^{-1}[\xi_{1}] ) = \infty.$
\end{claim*}
We prove this claim.
Firstly, one may observe that $\# \pi_2(\bigcup_{n \in \N} \tilde{f}^{-nk}(\rho, z_1)) = \infty$.
Indeed, if $\# \pi_2(\bigcup_{n \in \N} \tilde{f}^{-nk}(\rho, z_1)) < \infty$, then we have $\# \bigcup_{g \in H^{i}_{i}(S)} g^{-1}(z_1) < \infty$,
which implies the contradiction that $z_{1} \in \mathcal{E}_{i}(S)$.

Next, assume for contradiction that $\# \pi_2(\bigcup_{n \in \N} \tilde{f}^{-nk}(\rho, z_1) \cap \pi_1^{-1}([\xi_{1}]) ) < \infty$.
Noting that $i := i(e_{\tau,1})$ and the topological exactness of $\sigma^{k}$, the definition of $\mathcal{E}_{i}(S)$ and $k \in \N$,  we see that this inequality implies $z_1 \in \mathcal{E}_{i}(S)$.
Hence, our claim holds.
Therefore, the desired result follows from this claim.
            \end{proof}

\begin{thm}\label{73dayb4evI}
Let $S$ be a non-elementary, irreducible and aperiodic rational GDMS. 
Suppose that $\mathcal{E}_{i}(S)\subset F_{i}(S)$.
Let $K$ be a compact subset of $X(S)\times \widehat{\C}$ such that $K\cap (X_{i}(S)\times \mathcal{E}_{i}(S))=\emptyset $ for each $i\in V$.
Let $U$ be an open subset of $X(S)\times \widehat{\C}$ such that $U\cap J(\tilde{f})\neq \emptyset.$ 
Then, there exists an element $N\in \N$ such that for each $n\in \N$ with $n\geq N$, we have $\tilde{f}^{n}(U)\supset K.$ 
\end{thm}

\begin{proof}
By the irreducibility of $S$, for each $i, j\in V$, 
and for each $h\in H_{i}^{j}(S)$, we have 
$h(\mathcal{E}_{i}(S))\subset \mathcal{E}_{j}(S)$ and $h(F_{i}(S))\subset F_{j}(S).$ 
For each connected component $W$ of $F_{i}(S)$ with 
$W\cap \mathcal{E}_{i}(S)\neq \emptyset $, we take the hyperbolic metric. 
For each $r>0$, we denote by 
$A_{i}(r)$ the $r$-open neighborhood of $\mathcal{E}_{i}(S)$ in $F_{i}(S)$ with respect to the hyperbolic metric. Then, for each $i, j\in V$ and for each 
$h\in H_{i}^{j}(S)$, we have $h(A_{i}(r))\subset A_{j}(r).$ 
From the assumptions of our theorem, there exists a small $r>0$
 such that $K\cap ( X_{i}(S)\times A_{i}(r))=\emptyset $ for each 
 $i\in V.$ Thus $\tilde{f}(\cup _{i\in V} X_{i}(S)\times 
 (\widehat{\C}\setminus A_{i}(r)))\supset \cup _{i\in V} X_{i}(S)\times 
 (\widehat{\C}\setminus A_{i}(r)).$ 
 Hence, we may assume that $\tilde{f}(K)\supset K.$ 
 
 By Theorem \ref{density_repelling_per_points}, there exist a natural number $k$ and a 
 point $(\rho, a)\in U\cap J(\tilde{f})$ such that 
 $(\rho, a)$ is a repelling fixed point of $\tilde{f}^{k}.$ 
 Hence, there exists a 
 non-empty open neighborhood $U_{0}$ 
 of $(\rho, a)$ in $U$ such that $\tilde{f}^{k}(U_{0})\supset 
 U_{0}$ and $U_{0}\cap J(\tilde{f})\neq  \emptyset.$  
   
Furthermore, Lemma \ref{ftilde_nk_u_naught} shows that  $\tilde{f}^{nk}(U_{0})\supset K$ for any sufficiently large $n$.
Since $\tilde{f}^{l}(\tilde{f}^{nk}(U_{0})) \supset \tilde{f}^{l}(K) \supset K$ for each $l \in \N$,  there exists a number $N\in \N$ such that 
for each $n\in \N$ with $n\geq N$, we have 
$\tilde{f}^{n}(U)\supset K.$
\end{proof}

\begin{cor}\label{nnx6d2Yg0P}
Let $S$ be a non-elementary, irreducible and aperiodic rational GDMS. 
Suppose that $\mathcal{E}_{i}(S)\subset F_{i}(S)$ for each $i\in V$.
Then, $\tilde{f}: J(\tilde{f})\to J(\tilde{f})$ is topologically exact. 
\end{cor}

\begin{proof}
From the assumptions of our Corollary, 
$J(\tilde{f})\cap (X_{i}(S)\times \mathcal{E}_{i}(S))=\emptyset $ for each $i\in V$. 
Thus by Theorem \ref{73dayb4evI}, $\tilde{f}: J(\tilde{f})\to J(\tilde{f})$ is topologically exact. 
\end{proof}

\begin{cor}\label{ghraCYWJEm}
Let $S$ be a non-elementary, irreducible and aperiodic rational GDMS. 
Suppose that the associated skew product map $\tilde{f}$ is expanding along fibers.
Then, $\mathcal{E}_{i}(S)\subset F_{i}(S)$ for each $i\in V$ and thus $\tilde{f}: J(\tilde{f})\to J(\tilde{f})$ is topologically exact. 
\end{cor}

\begin{proof}
Under the assumptions of our corollary, suppose that 
there exists an $i\in V$ such that 
$\mathcal{E}_{i}(S)\cap J_{i}(S)\neq \emptyset .$ 
Let $x\in \mathcal{E}_{i}(S)\cap J_{i}(S).$ 
Then for each $h\in H^{i}_{i}(S)$, $x\in E(h)$ and 
$h^{2}(x)=x.$ We now show the following claim. 

\noindent 
\begin{claim*}\label{x4GD0cRUHa}
For each $h\in H_{i}^{i}(S)$ we have  $\deg (h)=1.$ 
\end{claim*}

To prove this claim, suppose there exists an $h\in H_{i}^{i}(S)$ with $\deg (h)\geq 2$. Then since $x\in E(h)$ and $\deg (h)\geq 2$, 
$x$ is a super attracting fixed point of $h^{2}$ \cite{beardon2000iteration}.
Since $S$ is non-elementary, $J_{i}(S)=J(H_{i}^{i}(S))$ is equal 
to the closure of the set of all repelling fixed points of elements of 
$H_{i}^{i}(S).$ Thus $\pi_{1} (J_{i}(\tilde{f}))=J_{i}(S)$, 
where $J_{i}(\tilde{f})=\{ (\xi, z) \in J(\tilde{f})\mid i(e_{\xi, 1})=i\}.$ 
In particular, there exists an element $\xi \in X_{i}(S)$ such that 
$(\xi ,x)\in J(\tilde{f})$. Let  $\rho =(g_{j}, e_{j})_{j=1}^{r}\in (\Rat\times E)^{r}$ be a finite 
admissible element 
such that $g_{r}\circ \cdots \circ g_{1}=h^{2}.$ Then 
$(\rho \xi, x)\in \tilde{f}^{-r}(J(\tilde{f}))=J(\tilde{f}).$ 
Moreover, since $x$ is a super attracting fixed point of 
$h^{2}$, we have $(\tilde{f}^{r})'(\rho \xi, x)=0.$ 
However, this contradicts that $\tilde{f}$ is expanding on 
$J(\tilde{f}).$ Thus we have proved Claim \ref{x4GD0cRUHa}. 

Since $S$ is non-elementary, 
$J_{i}(S)=J(H^{i}_{i}(S))$ is an infinite set and 
is equal to the closure of the set of all repelling fixed points of 
elements of $H_{i}^{i}(S).$ Thus there exist an element 
$y\in J_{i}(S)\setminus \mathcal{E}_{i}(S)$ and an element $h\in H_{i}^{i}(S)$ 
such that $y$ is a repelling fixed point of $h.$ 
Then $y$ is a repelling fixed point of $h^{2}.$ 
By Claim \ref{x4GD0cRUHa}, $h^{2}$ is a loxodromic M\"obius transformation, and 
$x$ is an attracting fixed point of $h^{2}.$ 
  
  By the argument in the proof of Claim \ref{x4GD0cRUHa}, there exists an element 
  $\xi \in X(S)$ such that $(\xi, x)\in J(\tilde{f})$.
  Since $\tilde{f}$ is expanding on $J(\tilde{f})$, there exists a number 
  $N\in \N $ such that for each $n \ge N$   and for each $(\alpha, z)\in J(\tilde{f})$, 
  we have
  \begin{equation}\label{wMLNpYCayV}
\| (\tilde{f}^{n})'(\alpha ,z)\| >2.
\end{equation}
  Let $\beta =(\gamma _{j}, e_{j})_{j=1}^{s}\in 
  (\Rat \times E)^{s}$ be a finite admissible element 
  such that $\gamma _{s}\circ \cdots \circ \gamma _{1}=h^{2N}.$ 
  Then $(\beta \xi, x)\in (\tilde{f}^{s})^{-1}(J(\tilde{f}))=J(\tilde{f})$. 
  Moreover, since $x$ is an attracting fixed point of $h^{2}$, 
  we have $\| (\tilde{f}^{s})'(\beta \xi, x)\| <1.$  
However, this contradicts (\ref{wMLNpYCayV}).  

Therefore $\mathcal{E}_{i}(S)\cap J_{i}(S)= \emptyset $ for each $i\in V.$ 
By Theorem \ref{73dayb4evI}, it follows that 
$\tilde{f}$ is topologically exact on $J(\tilde{f}).$  
Hence, we have proved our corollary. 
\end{proof}

\subsection{Hyperbolicity implies expandingness}\label{1FV2xMdKQT}In this subsection, we shall prove that Main result C by verifying that the hyperbolic system $S$ satisfies the assumption of Theorem 2.17 of \cite{MR1827119}, which is an assertion to the same effect as ours. 
The only task is to prove that under certain assumptions a rational GDMS satisfies a condition called Condition (C2) given below. 
To clarify what this condition is, let us present the setting in which it was introduced.

\begin{dfn}\cite[Definition 2.1]{MR1827119}\label{T4fhvGEuO4} 
Let $(X,d)$ be a compact metric space. 
Suppose that $p: X \to X$ is a continuous map and $q_x: \widehat{\C} \to \widehat{\C}$ is a rational map with a degree at least one for each $x \in X$.
Then, a map $\tilde{f}: X \times \widehat{\C} \to X \times \widehat{\C}$ given by $\tilde{f}((x, y)):=(p(x), q_x(y))$ is called \textit{a rational skew product}, where $p: X \to X$ is a continuous map and $q_x:\widehat{\C} \to\widehat{\C}$ is a rational map with degree at least one for each $x \in X$.
Then we say that $\tilde{f}: X \times\widehat{\C} \to X \times\widehat{\C}$ is a rational skew product.
For each $n \in \N$ and $x \in X$, we set $q_x^{(n)}:=q_{p^{n-1}(x)} \circ \cdots \circ q_x$ and $\tilde{f}_x^n := \tilde{f}^n|_{\pi_X^{-1}(x)}$. 
For each $x \in X$, let $F_x  :=\{y \in \widehat{\C} : (q_x^{(n)})_{n \in \N} \text { is normal in a neighborhood of } y\}$ and $J_{\xi} := \widehat{\C} \setminus F_x$ and $\tilde{J}_x=\{x\} \times J_{\xi}$.
Further, we set $ J(\tilde{f}) := \overline{\bigcup_{x \in X} \tilde{J}_x} \subset X \times \widehat{\C}$ and $F(\tilde{f}):=(X \times \widehat{\C}) \setminus J(\tilde{f})$.
We endow the set of all non-empty subsets of $\widehat{\C}$ with the Hausdorff metric.
\end{dfn}

\begin{dfn}\label{VswE1URs0W}\cite[Condition (C2)]{MR1827119}
We say that a rational skew product $\tilde{f}$ satisfies Condition (C2) if for each $x_{0} \in X$ there exists an open neighborhood $O$ of $x_{0}$ and a family $(D_x)_{x \in O}$ of round discs in $\widehat{\C}$ such that the following three conditions are satisfied:
\begin{enumerate}[label=(\roman*)]
\item\label{Xckow2qKrc} $\overline{\bigcup_{x \in O} \bigcup_{n \in \N_{0}} \tilde{f}^n(\{x\} \times D_x)} \subset F(\tilde{f})$;
\item\label{GOJn2gojWL} for any $x \in O$, we have that $\diam(q_x^{(n)}(D_x)) \to 0$ as $n \to \infty$;
\item\label{iCxAGUDljH} $x \mapsto D_x$ is continuous in $O$.
\end{enumerate}
\end{dfn}

\begin{dfn}\cite{MR4002398} 
Let $L_i$ be a subset of $\widehat{\C}$ for each $i \in V$.
For a family $\mathcal{F} \subset \Rat$ and a set $A \subset \widehat{\C}$, we write $\mathcal{F}(A):=\bigcup_{f \in \mathcal{F}} f(A)$ and $\mathcal{F}^{-1}(A):=\bigcup_{f \in \mathcal{F}} f^{-1}(A)$, where we set $\mathcal{F}(A):=\emptyset$ and  $\mathcal{F}^{-1}(A):=\emptyset$ if $\mathcal{F}=\emptyset$. 
Utilizing these symbols, we name the following properties.
\begin{enumerate}[label=(\roman*)]
\item The family $\left(L_i\right)_{i \in V}$ is said to be \textit{forward $S$-invariant} if $\Gamma_e\left(L_{i(e)}\right) \subset L_{t(e)}$ for each $e \in E$.
\item The family $\left(L_i\right)_{i \in V}$ is said to be \textit{backward $S$-invariant} if $\Gamma_e^{-1}\left(L_{t(e)}\right) \subset L_{i(e)}$ for each $e \in E$.
\end{enumerate}
\end{dfn}

\begin{prp}  \label{oqUbNFYo4a} 
Let $S$ be a rational GDMS. Then, the families $(F_i(S))_{i \in V}$ and $(P_{i}(S))_{i \in V}$ are both forward $S$-invariant and the family $(J_i(S))_{i \in V}$ is backward $S$-invariant.
\end{prp}
\begin{proof}
The forward $S$-invariance of $(P_{i}(S))_{i \in V}$ follows from the definition. See Lemma 2.15 of \cite{MR4002398} for the other assertions. 
\end{proof}

\begin{prp}\label{Oqeiwu19LMc}
Let $S := \left(V, E,\left(\Gamma_e\right)_{e \in E}\right)$ be a non-elementary, hyperbolic, irreducible rational GDMS.
 For each $i \in V$ and $h \in H_{i}^{i}(S)$, assume that $h$ is loxodromic when $\deg(h)= 1$.
 Then, the skew product map $\tilde{f}$ of $S$ satisfies Condition (C2).
\end{prp}

\begin{proof}
Let  $i \in V$. 
We use $\mathcal{W}_{i} $ to denote the set of connected components of $F_i(S)$ which have a non-empty intersection with $P_i(S)$. Note that the hyperbolicity and irreducibility of $S$ imply $\mathcal{W}_{i} \neq \emptyset$. The compactness of $P_{i}(S)$ in $\widehat\C$ guarantees that $\# \mathcal{W}_{i} < \infty$. We use  $\Omega_{i}$ as an index set for the elements of $\mathcal{W}_{i}$ and we write  $\mathcal{W}_{i} = \{W_{\omega} : \omega \in \Omega_{i}\}$.
For each $\omega \in \Omega_{i}$, we may take some compact subset $K_{i, \omega}$ of $W_{\omega}$ given by
$K_{i, \omega} := \{\zeta \in W_{\omega} : d_{\omega}(\zeta, P_{i}(S) \cap W_{\omega}) \leq 1\}$,
where $d_{\omega}$ denotes the hyperbolic distance on $W_{\omega}$ and $d_{\omega}(\zeta, P_{i}(S) \cap W_{\omega}) := \inf\{d_{\omega}(\zeta, p): p \in P_{i}(S) \cap W_{\omega}\}$ for each $\zeta \in W_{\omega}$. 
We shall write $\mathcal{K} :=\{ K_{i, \omega} :  i \in V, \omega \in \Omega_{i}\}$.

Here, for each $\xi:=(\xi_{n})_{n \in \N} \in X(S)$, take a family of spherical open discs $(D_{\rho})_{\rho \in [\xi_{1}]}$ such that for each $\rho, \rho^{\prime} \in [\xi_{1}]$ we have $D_{\rho} = D_{\rho^{\prime}}$, $D_{\rho} \cap P_{i(e_{\xi,1})}(S) \neq \emptyset$ and $D_{\rho} \subset K_{i(e_{\xi,1}),\omega}$ for some $\omega \in \Omega_{i(e_{\xi,1})}$.
It is immediately clear from definitions that this finite family satisfies \ref{iCxAGUDljH} of Condition  (C2).  
Moreover, {noting $D_\rho \subset K_{i(e_{\xi,1}),\omega} \subset W_{\omega}$}, we see that $(D_{\rho})_{\rho \in [\xi_{1}]}$ satisfies \ref{Xckow2qKrc} of Condition  (C2) for any $\xi \in X(S)$. 
Now, \ref{GOJn2gojWL} remains to be verified for $(D_{\rho})_{\rho \in [\xi_{1}]}$. For simplicity, we shall write $N := \# V \times \max_{j \in V}\# W_{j}$ and $g_{\xi,[a, b]} := g_{\xi, b} \circ \cdots \circ g_{\xi, a}$ for every $\xi \in X(S)$ and $a, b \in \N$ with $a < b$.

Now, we fix $\eta > 0$, $\xi \in X(S)$, a member $D_{\rho}$ of $(D_{\rho})_{\rho \in [\xi_{1}]}$ for the remainder of the proof.
Take arbitrary $z \in D_{\rho} \cap P_{i(e_{\xi,1})}(S)$ and $k^{\prime} \in \N$. 
Then, the $S$-forward invariance of $P(S)$, $\# \mathcal{W}_{i} < \infty$ and $\# V < \infty$ imply that there exists $(i,\omega, (a, b)) \in V \times \Omega_{i} \times \N^2$ such that 
\begin{equation}\label{fWb770CcN0}
k^{\prime} \leq b, a < b \text{ and  }g_{\xi, [1,b]}(z), g_{\xi, [1,a]}(z) \in P_{i}(S) \cap W_{\omega}.
\end{equation}

We shall verify the next claim under the setting above.

\begin{claim*}
For each $k^{\prime} \in \N$ and $(i,\omega, (a, b)) \in V \times \Omega_{i} \times \N^2$ satisfying (\ref{fWb770CcN0}), we have the following.
\begin{enumerate}[label=(\roman*)]
\item \label{7yrcYq1uP2} $g_{\xi, [a, b]}(W_{\omega}\cap P_{i}(S)) \subset W_{\omega} \cap P_{i}(S)$.
\item \label{cbGMu7CyrB} there exists $\lambda (i, K_{i,\omega}, g_{\xi, [a,b]}) \in (0,1)$ such that 
for all  $z_{1}, z_{2} \in K_{i,\omega}$ we have
\begin{equation}\label{297jHDAKc}
d_{\omega}(g_{\xi, [a,b]} (z_{1}), g_{\xi, [a,b]} (z_{2})) 
\leq \lambda (i, K_{i,\omega}, g_{\xi, [a,b]}) d_{\omega}(z_{1}, z_{2}).
\end{equation}
\end{enumerate}
\end{claim*}
We prove this claim.
As \ref{7yrcYq1uP2} follows immediately from definitions, we shall prove \ref{cbGMu7CyrB}.
Suppose that the map $g_{\xi, [a, b]}$ satisfies $\deg(g_{\xi, [a, b]}) \geq 2$.
Remark that the hyperbolicity of $S$ and definitions imply that $P(g_{\xi, [a,b]}) \subset P_{i}(S) \subset F_{i}(S) \subset F(g_{\xi, [a,b]})$.
Noting \ref{7yrcYq1uP2}, one may prove that $W_{\omega}\cap P_{i}(S)$ contains an attracting fixed point of $g_{\xi, [a,b]}$.
Hence,  by the Pick theorem \cite{milnor2011dynamics}, there exists $\lambda (i, K_{i,\omega}, g_{\xi, [a,b]}) \in (0,1)$ such that 
for all  $z_{1}, z_{2} \in K_{i,\omega}$, we have 
\begin{equation}\label{297jHDAKc}
d_{\omega}(g_{\xi, [a,b]} (z_{1}), g_{\xi, [a,b]} (z_{2})) 
\leq \lambda (i, K_{i,\omega}, g_{\xi, [a,b]}) d_{\omega}(z_{1}, z_{2}).
\end{equation}
Next, we assume that $\deg(g_{\xi, [a,b]}) = 1$.
Since, $g_{\xi, [a,b]}$ is loxodromic by the assumption,  $W_{\omega} \cap P_{i}(S)$ contains an attracting fixed point of $g_{\xi, [a,b]}$.
Thus, the Pick theorem shows that there exists $\lambda (i, K_{i,\omega}, g_{\xi, [a,b]}) \in (0,1)$ such that for all  $z_{1}, z_{2} \in K_{i,\omega}$ the inequality in (\ref{297jHDAKc}) holds.
Therefore, the conclusion of our claim follows.

For each natural number $l \leq  N$, we set
$$
\mathcal{H}_{i, \xi}(l) := \{g_{\xi, [a,b]} \in H_{i}^{i}(S):  \exists (\omega, (a, b)) \in \Omega_{i} \times \N^2 \text{ s.t. } b - a = l, g_{\xi, [1,b]}(z), g_{\xi, [1,a]}(z) \in P_{i}(S) \cap W_{\omega} \}.
$$
Since $V$ and $\Gamma$ are both finite sets , $\mathcal{H}_{i, \xi}(l)$ is a finite set and thus
$$
\lambda_{\xi} := \max\left\{ \lambda (j, K_{j,\omega}, h) : i \in V, \omega \in \Omega_{j}, K_{j,\omega} \in \mathcal{K}, h \in \bigcup_{l = 1}^{N} \mathcal{H}_{j, \xi}(l) \neq \emptyset \right\}< 1.
$$
Also, letting $\diam_{\omega}$ denote the diameter in the hyperbolic distance $d_{\omega}$ for $\omega \in \Omega$, we have
$$
c := \sup \left\{\frac{\diam_{\omega}(A)}{\diam_{\widehat\C}(A)} :  j \in V, \omega \in \Omega_{j}, A \subset K_{j,\omega} \text{ is compact in } W_{\omega} \right\} < \infty,
$$
where we used the fact that the ratio of conformal metrics is bounded away from 0 and infinity when restricted on a compact set.

Let $k\in \N$ be such that $c\lambda_{\xi}^{k} < \eta$.
Then, it follows from the $S$-forward invariance of $P(S)$, $\# \mathcal{W}_{i} < \infty$ and $\# V < \infty$  that for all $n \geq k(N + 1)$, there exist natural numbers $\alpha_{1}, \alpha_{2}, \ldots,  \alpha_{k}, \beta_{1}, \beta_{2}, \ldots,  \beta_{k}$ such that $1 \leq \alpha_{1} < \beta_{1} < \alpha_{2} < \beta_{2} < \cdots < \alpha_{k} < \beta_{k} \leq n$ satisfying the property that for each $l \in \{1, \ldots, k\}$ there exist $i \in V$ and $\omega \in \Omega_{i}$ such that $g_{\xi, [1,\beta_{l}]}(z)$, $g_{\xi, [1,\alpha_{l}]}(z) \in P_{i}(S) \cap W_{\omega}$.
Thus, we obtain
\begin{equation}\label{4EsQQKSSKc}
\diam_{\widehat\C}\left(g_{\xi, [1,n]}(D_{\rho})\right)
\leq c  \lambda_{\xi}^{k} < \eta.
\end{equation}
Combining this with \ref{oqUbNFYo4a} of Proposition \ref{UHyK6FOUuf} and the fact that $\xi \in X(S)$ and $z \in P_{i(e_{\xi,1})}(S)$ are arbitrary, we may conclude that (2) of Condition (C2) holds by taking $(D_{\rho})_{\rho \in [\xi_{1}]}$ for each $\xi \in X(S)$. 
Consequently, the desired statement follows.
\end{proof}

\begin{thm}\label{A7J66dnGYg}
Let $S$ be a hyperbolic, non-elementary, irreducible rational GDMS.
 For each $i \in V$ and $h \in H_{i}^{i}(S)$, assume that $h$ is loxodromic when $\deg(h)= 1$.
 Then, the skew product map $\tilde{f}$ of $S$ is expanding along fibers.
\end{thm}
\begin{proof}
Since $\tilde{f}$ satisfies Condition (C2), Theorem 2.17 of \cite{MR1827119} verifies the statement.
\end{proof}

\subsection{Bowen's formula}\label{qdXaxWWPnV}
The last subsection is dedicated to presenting the central result of Bowen's formula for a rational GDMS.
We shall start by proving the existence of a conformal measure for $(J(\tilde{f}), \tilde{f}|_{J(\tilde{f})})$ and deriving some properties.
  
We shall write $\tilde{f}:= \tilde{f}|_{J(\tilde{f})}$ to avoid redundancy when the identification is clear from the context.
Denote by $C(J(\tilde{f}))$ the space of all real-valued, Borel measurable functions on $J(\tilde{f})$ equipped with the supremum norm written as $\left\| \cdot \right\|_{J(\tilde{f})}$.
Let $\varphi: J(\tilde{f}) \to \R$ be a continuous function and $P(\tilde{f}, \varphi)$ be the topological pressure of $(\tilde{f}, \varphi)$.
Then, the variational principle holds: $P(\tilde{f}, \varphi) = \sup\{h_\mu(\tilde{f})+\int\varphi d \mu : \mu \in M_{\tilde{f}}(J(\tilde{f}))\}$, where $M_{\tilde{f}}(J(\tilde{f}))$ is the set of all $\tilde{f}$-invariant Borel probability measures on $J(\tilde{f})$ and  $h_\mu(\tilde{f})$ is the measure-theoretic entropy of $\tilde{f}$ with respect to $\mu \in M_{\tilde{f}}(J(\tilde{f}))$; see \cite{MR648108}.
An element $\mu \in M_{\tilde{f}}(J(\tilde{f}))$ is called an equilibrium state of $(\tilde{f}, \varphi)$ if $P(\tilde{f}, \varphi) = h_\mu(\tilde{f})+\int\varphi d \mu$.

When the skew product Julia set $J(\tilde{f})$ contains no critical point of $\tilde{f}$, we define a function $\tilde\varphi \colon J(\tilde{f}) \to \R$ given by $\tilde\varphi(\tilde{z}):= -\log \|g_{\xi_{1}}^{\prime}(z)\|$ for each $\tilde{z} := (\xi, z) \in J(\tilde{f})$.
Let $t \in \R$.
Then, a Perron-Frobenius operator $\mathcal{L}_{t\tilde\varphi} \colon C(J(\tilde{f})) \to C(J(\tilde{f}))$ is defined by
$$
\mathcal{L}_{t\tilde\varphi} \psi(\tilde{z}) := \sum_{\tilde{y} \in \tilde{f}^{-1}(\{\tilde{z}\})} \exp{(t\tilde\varphi(\tilde{y}))}\psi(\tilde{y})
$$
for each $\psi \in C(J(\tilde{f}))$ and $\tilde{z} \in J(\tilde{f})$.
Let $\mathcal{L}_{t\tilde{\varphi}}^*$ denote the dual operator of $\mathcal{L}_{t\tilde{\varphi}}$ and $\tilde{\mathcal{F}}$ denote the Borel $\sigma$-algebra on $J(\tilde{f})$.
Also, a function $\mathcal{P}  \colon \R \to \R\cup \{\infty\}$ defined by $\mathcal{P} (t) :=P(\tilde{f}, t\tilde\varphi)$ is called the geometric pressure function of $(\tilde{f}, \varphi)$.
All the statements of this subsection shall be given in this setting.
\begin{lem}\label{qQrOdN2Nmt}
Suppose that a rational GDMS $S$ is expanding.
Then, the following hold.
\begin{enumerate}[label=(\roman*)]
\item \label{t7xqn180I5} For each $t \in \R$, there exists a unique probability measure $\tilde{\nu}_{t}$ on $(J(\tilde{f}), \tilde{\mathcal{F}})$ such that $\mathcal{L}_{t\tilde\varphi}^* \tilde{\nu}_{t} =  \lambda\tilde{\nu}_{t}$, where $\lambda := \mathcal{L}_{t\tilde\varphi}^*(\tilde{\nu}_{t})(1)$.
\item \label{nJEgDfnVpr} There exists $h \in C(J(\tilde{f}))$ such that $\mathcal{L}_{t\tilde\varphi}h = \lambda h$, $\tilde\nu_{t}(h)=1$, $h>0$  and 
$$
\lim_{n \to \infty}\left\|\lambda^{-n} \mathcal{L}_{t\tilde\varphi}^n \psi-\tilde{\nu}_{t}(\psi) h\right\|_{J(\tilde{f})} = 0
$$
for each $\psi \in C(\tilde{J}(\tilde{f}))$.
\item \label{Lg4K3pcyFt} The measure $h \tilde{\nu}_{t}$ is $\tilde{f}$-invariant, exact (hence ergodic) and an equilibrium state for $(\tilde{f}, t\tilde{\varphi})$.
\item  \label{cmu6vrJlgl}  There exists a unique zero $\delta$ of $t \mapsto \mathcal{P}(t)$ such that $\delta = -h_{h \tilde{\nu}_{\delta}}(\tilde{f})/\int \tilde{\varphi} h d \tilde{\nu}_{\delta}$.
\end{enumerate}
\end{lem}
\begin{proof}
Noting the topological exactness of $\tilde{f}$ proved in Corollary \ref{ghraCYWJEm} and its expandingness, one may prove that a sufficiently high iteration $\tilde{f}^{s}$ satisfies Condition I and II of \cite{MR648108}.
Thus,  \ref{t7xqn180I5} follows from the application to $(\tilde{f}^{s}, \tilde\varphi)$ of Lemma 2 in \cite{MR648108}, and \ref{nJEgDfnVpr} also follows from Theorem 8 in \cite{MR648108}. One may justify \ref{Lg4K3pcyFt} by the argument on page 140 of \cite{MR648108}.

To prove \ref{cmu6vrJlgl}, consider a rational semigroup generated by all the elements of  $\Gamma$, and its skew product Julia set $J_{\Gamma}(\tilde{f})$. 
Since $h_{\text {top}}(\tilde{f}|_{J(\tilde{f})}) \leq h_{\text{top}}(\tilde{f}|_{J_{\Gamma}(\tilde{f})})$, it follows from Theorem 6.1 of \cite{MR1767945} that $h_{\text {top}}(\tilde{f}|_{J(\tilde{f})}) \leq h_{\text{top}}(\tilde{f}|_{J_{\Gamma}(\tilde{f})}) \leq \log (\sum_{g \in \Gamma} \deg(g))$, where $h_{\text {top}}(\tilde{f}|_{J(\tilde{f})})$ and $h_{\text{top}}(\tilde{f}|_{J_{\Gamma}(\tilde{f})})$ refer to the topological entropies of $\tilde{f}|_{J(\tilde{f})}$ and $\tilde{f}|_{J_{\Gamma}(\tilde{f})}$.
Combining this fact with the expandingness of $\tilde{f}$ verifies the function $\mathcal{P} $ is strictly decreasing on $\R$, and $\lim_{t \to \infty}\mathcal{P} (t) = -\infty$ and $\lim_{t \to -\infty}\mathcal{P} (t) = \infty$ hold. 
The last identity may be justified by the definition of the equilibrium state $h \tilde{\nu}$ of $(\tilde{f}, \delta\tilde\varphi)$.
\end{proof}

Let $\mathscr{L}^1(\tilde{\nu}) := \{f: J(\tilde{f}) \to \R : \int_{J(\tilde{f})} |f| d \tilde{\nu} < \infty \}$  and $L^1(\tilde{\nu}) := \mathscr{L}^1(\tilde{\nu})/\sim$ equipped with $L^1$ norm, where the relation $f \sim g$ means $f(\tilde{z})=g(\tilde{z})$ for $\tilde{\nu}$-a.e. $\tilde{z} \in J(\tilde{f})$.

\begin{lem}\label{tg3o75180I5}
For an expanding rational GDMS $S$, consider the probability space $(J(\tilde{f}), \tilde{\mathcal{F}}, \tilde{\nu}_{\delta})$.
Let $k \in \N$.
Then, each $A \in \tilde{\mathcal{F}}$ such that $\tilde{f}^k|_{A}: A \to \tilde{f}^k(A)$ is injective, we have 
$$
\tilde{\nu}_{\delta}(\tilde{f}^k(A))=\int_A\|(\tilde{f}^k)^{\prime}\|^\delta d \tilde{\nu}_{\delta}.
$$
\end{lem}
\begin{proof}
One may first show that  $\mathcal{L}_{\delta\tilde{\varphi}} \psi \in L^1(\tilde{\nu}_{\delta})$ for $\psi \in L^1(\tilde{\nu}_{\delta})$ in the same manner as Lemma 3.9 in \cite{MR2153926}. 
Remark that $\tilde{f}$ is finite-to-one.
Then, following the approach taken in Lemma 3.10 in \cite{MR2153926}, we see that Proposition 2.2 in  \cite{MR1014246} verifies the claim.
\end{proof}
Adopting a similar approach in the proof of Lemma 4.3 in \cite{MR2153926}, we obtain the following.
\begin{prp}\label{3ni8tG7iYu}
Let $S$ be an expanding rational GDMS. 
Let $m$ be a probability measure on $(J(\tilde{f}), \tilde{\mathcal{F}})$ and $V$ be an open set in $X(S) \times \widehat{\C}$ such that $V \cap J(\tilde{f}) \neq \emptyset$.  
Suppose that 
\begin{enumerate}[label=(\roman*)]
\item $\tilde{f}^n|_{V}: V \to \tilde{f}^n(V)$ is a homeomorphism;
\item for any Borel set $A \subset J(\tilde{f}) \cap V$, $
m(\tilde{f}^n(A))=\int_A\|(\tilde{f}^n)^{\prime}\|^t d m$,
\end{enumerate}
and let  $\eta\colon g(V) \to V$ be an inverse mapping of  $\tilde{f}^n|_{V}$. 
Then, we have  $m(\eta(B))=\int_B\|(\tilde{f}^n)^{\prime}(\eta)\|^{-t} d m$ for any Borel set $B$ in $\tilde{f}^n(V) \cap J(\tilde{f})$.
\end{prp}

Now, we may define a push-forward measure $\nu:= (\pi_{2}|_{J(\tilde{f})})_*\left(\tilde{\nu}_{\delta}\right)$ under the restriction $\pi_{2}|_{J(\tilde{f})}$.
Let $\spt(m)$ denote the support of a Borel probability measure $m$ on $J(\tilde{f})$.

\begin{prp}\label{YW4rYXJmii}
Let $S$ be a non-elementary, irreducible rational GDMS. 
Suppose that $\tilde{f}$ is expanding and the incidence matrix is aperiodic.
Then we have $\spt (\tilde{\nu}_{\delta}) = J(\tilde f)$ and $\spt(\nu)=J(S)$.
\end{prp}
           \begin{proof}
 The first conclusion follows from the topological exactness of $\tilde{f}$ and the conformality of $\tilde{\nu}_{\delta}$.
Since $\pi_2 (J(\tilde{f}))=J(S)$, we obtain $\spt(\nu) = J(S)$.
\end{proof}

            Following \cite[Theorem 8.3.12]{MR4530063} and  \cite[Lemma 8.3.14 ]{MR4530063}, we obtain the results based on the Koebe distortion theorem.
\begin{lem}\label{koebe1}
 Let $R, s \in(0, \pi)$.
 Then, there exists a monotone increasing continuous function $k \colon [0,1) \to[1, \infty)$ such that for each $z \in \widehat{\C}$, $r \in(0, R]$, $t \in[0,1)$ and univalent analytic function $h: D(z, r) \to \widehat{\C}$ with $D \subset \widehat{\C} \setminus h(D(z, r))$, we have
 $$
 \sup \{\|h^{\prime}(z)\|: z \in D(z, t r)\} \leq k(t) \inf \{\|h^{\prime}(z)\|: z \in D(z, t r)\}.
 $$
\end{lem}

\begin{lem}\label{koebe2}
Let $z \in \widehat{\C}$ and $R,s \in 0,(0, \pi)$.
Let $h\colon D(z, 2 R) \to \widehat{\C}$ be a univalent analytic function such that  $h(D(z, 2 r)) \cap D \neq \emptyset$ with some radius $s \in(0, \pi)$.
Then, for each $0 \leq r \leq R$ we have
$$
D(h(z), k^{-1}(2^{-1}) r\|h^{\prime}(z)\|)  \subset h(D(z, r)) \subset D(h(z), k(2^{-1}) r\|h^{\prime}(z)\|).
$$
\end{lem}

Let $E \subset \widehat{\C}$, $u \geq 0$.
We shall denote by $\mathcal{H}^u(E)$ and $\dim_H(E)$  respectively the $u$-dimensional (outer) Hausdorff measure and Hausdorff dimension of $E$ with respect to the distance $d_{\widehat{\C}}$.

To prove the next two theorems, we shall adopt approaches similar to the ones developed in \cite[Theorem 3.4]{sumi1998hausdorff} and  \cite[Lemma 3.16]{MR2153926}.
\begin{dfn}\cite{MR4002398}\label{wVNh1F6ao0}
We say a rational GDMS satisfies the \textit{backward separating condition} (BSC)  if 
$$
g_1^{-1}\left(J_{t\left(e_1\right)}(S)\right) \cap g_2^{-1}\left(J_{t\left(e_2\right)}(S)\right)=\emptyset
$$ for every $e_1, e_2 \in E$ with the same initial vertex and for every $g_1 \in \Gamma_{e_{1}}, g_2 \in \Gamma_{e_{2}}$, except the case $e_1=e_2$ and $g_1=g_2$.
\end{dfn}

\begin{thm}\label{thm_dimension_julia_delta}
  Let $S$ be a non-elementary, expanding, irreducible and aperiodic rational GDMS satisfying the backward separating condition. 
    Then we have $\mathcal{H}^{\delta}(J(S)) > 0.$
\end{thm}

\begin{proof}
By the same arguments as in  \cite[Proof of Theorem 2.8]{sumi1998hausdorff}  we have that $S$ is hyperbolic.
 Thus, we may take some $\varepsilon \in (0,1)$ such that for all $j \in V$ there exists a spherical neighborhood $\mathcal{N}(J_{j}(S), \varepsilon)$ of $J_{j}(S)$ with $\mathcal{N}(J_{j}(S), \varepsilon) \cap P_{j}(S) = \emptyset$.
For each $j \in V$, we define
\begin{align*}
\rho_{j} := &\min\{\dist_{\widehat\C}\left(g_{1}^{-1}(J_{t(e_{1})}(S)), g_{2}^{-1}(J_{t(e_{2})}(S) )\right)\\
& : e_{1}, e_{2} \in E \text{ with } i(e_{1}) = i(e_{2}) = j, \text{ and } g_{1} \in \Gamma_{e_{1}}, g_{2} \in \Gamma_{e_{2}} \text{ except for  $g_{1} = g_{2}$, $e_{1} = e_{2}$} \}.
\end{align*}
By the BSC, $\rho_{j} > 0$ for each $j \in V$, and thus $\rho := \min_{j \in V} \rho_{j}>0$. 

Fix $\tilde{z}:=(\xi,z) \in J(\tilde{f})$ and $r \in(0, \varepsilon)$, and let $i := i(e_{\xi,1})$. 
Remark that $z \in J_i(S)$.
Then,  \ref{sc15xYijnn} in Proposition \ref{UHyK6FOUuf} and \ref{projection} in Proposition \ref{UHyK6FOUuf} show that for each $n \in \N$ there exists $\xi|_{n}\in X^{n}(S)$ such that 
$z_{n} := g_{\xi|_{n}}(z) = \pi_{2}(\tilde{f}^{n}(\tilde{z})) \in J_{t(e_{\xi, n})}(S)$.
Also, by the analytic continuation, there exists a holomorphic map $\gamma_{\xi|_{n}}$ on $D(z_{n}, \varepsilon)$, which is the inverse mapping of the restriction of $g_{\xi|_{n}}$ to $g^{-1}_{\xi|_{n}}(D(z_{n}, \varepsilon))$.
We shall write $r_{n} := 2^{-1}\varepsilon \|\gamma_{\xi|_{n}}^{\prime}(z_{n})\| > 0$ and remark that $\lim_{n \to \infty}r_{n} = 0$ because $\tilde f$ is expanding along fibers. 

Besides, we use Lemma \ref{koebe2} to ensure that for all $n \in \N$ we have
\begin{equation}\label{inclusion_eta_1}
D(z, K^{-1}r_{n})
\subset \gamma_{\xi|_{n}}(D(z_{n}, 2^{-1}\varepsilon)) 
\subset D(z, Kr_{n}),
\end{equation}
where $K := k(2^{-1})$ is given in Lemma \ref{koebe2}. 
In addition, by Lemma \ref{koebe1}  we see that 
    \begin{equation}\label{enqlClDijr}
        \|\gamma_{\xi|_{n}}^{\prime}(\zeta)\| \leq 2K\varepsilon^{-1} r_{n}, 
    \end{equation}
for all $\zeta \in D(z_{n}, 2^{-1}\varepsilon)$. 

Furthermore, writing $\Gamma := \cup_{e \in E}\Gamma_{e}$, we have for all $n \in \N$ 
$$
\frac{r_{n}}{r_{n + 1}}
                = \frac{\|\gamma_{\xi|_{n}}^{\prime}(z_{n})\|}{\|\gamma_{\xi|_{n + 1}}^{\prime}(z_{n + 1})\|} 
\leq \sup_{g \in  \Gamma} \sup_{z \in g^{-1}(J(S))}\|g^{\prime}(z)\|
=: c_{M}.
$$
Letting  $n$ be the largest natural number such that
\begin{equation}\label{K4r63f7qhJ}
r_{n + 1} < Kr \leq r_{n},
\end{equation}
we obtain 
\begin{equation}\label{bhhlyOt8hr}
K_{1}^{-1} := K^{-1}c^{-1}_{M} 
< \frac{r}{r_{n + 1}}\frac{r_{n + 1}}{r_{n}}
=  rr^{-1}_{n}
\leq K^{-1}.
\end{equation}
By taking $\varepsilon>0$ sufficiently small, we may assume that $K^{-1}r_n = K^{-1}2^{-1}\varepsilon \|\gamma_{\xi|_{n}}^{\prime}(z_{n})\| < \rho$.
Note that since $\tilde f$ is expanding along fibers, we can choose $\varepsilon$ independently of the point $\tilde{z} \in J(\tilde f)$.

Next, we define  $\boldsymbol{D} := \{(\tau, \zeta) \in J(\tilde{f}) :  \tau \in [\xi_{1}\ldots, \xi_{n}], \zeta \in \gamma_{\xi|_{n}}\left(D(z_{n}, 2^{-1}\varepsilon) \right)\}$.
 Then, we see that  $\tilde{f}^{n}|_{\boldsymbol{D}} : \boldsymbol{D} \to \tilde{f}^{n}(\boldsymbol{D})$ is a homeomorphism with respect to the subspace topology on $\boldsymbol{D}$ and that we have
$$
\tilde{f}^{n}(\boldsymbol{D}) = \left\{(\tau, \zeta) \in J(\tilde{f})  : i(e_{\tau,1})=t(e_{\xi, n}), \zeta \in D(z_{n}, 2^{-1}\varepsilon) \right\}.
$$
Also, we let $\eta_{\xi|_{n}}$ denote some inverse branch of $\tilde{f}^{n}|_{\boldsymbol{D}}$ given by $\eta_{\xi|_{n}}((\tau, \zeta)) = (\xi|_n  \tau, \gamma_{\xi|_{n}}(\zeta))$ for each $(\tau, \zeta) \in \tilde{f}(D)$. 
    Now, let us prove the following claim under the setting developed in this argument.
\begin{claim*} 
\begin{equation}\label{inclusion_ita_from_below_new}
\pi_1^{-1}(X_i(S)) \cap \pi_{2}|_{J(\tilde{f})}^{-1}\left(D(z, K^{-1}r_{n})\right)  \subset \boldsymbol{D} .
\end{equation}
\end{claim*}
\begin{proof} 
Let $(\tau, \kappa)$ be an element of the LHS of (\ref{inclusion_ita_from_below_new}), where we write $\tau := (\tau_{1}, \tau_{2}, \ldots)$.
Then, it suffices to prove $\tau \in [\xi_{1}\ldots, \xi_{n}]$ due to (\ref{inclusion_eta_1}).
 Suppose for contradiction that $\tau_{1} \neq \xi_{1}$. 
 Remark that the definitions of $(\tau, \kappa)$ and $(\xi, z)$ imply that $i(e_{\xi,1})=i(e_{\tau,1})=i$, $\kappa \in g_{\tau_1}^{-1} (J_{t(e_{\tau,1})}(S))$, $z\in g_{\xi_1}^{-1}(J_{t(e_{\xi,1})}(S))$.
In addition,  $d_{\widehat\C}(\kappa, z)\le K^{-1} r_n < \rho$. 
However, this contradicts the fact that $d_{\widehat\C}(\kappa, z) \geq \rho$, which is guaranteed by the assumption of the BSC. 
Hence, we have $\tau_1=\xi_1$. Similarly, one may verify $\tau_k=\xi_k$ for all $k \in \N$ with $2\le k\le n$.
Our claim has been verified.
 \end{proof}
 Recall the Borel probability measure $\tilde \nu := \tilde \nu_{\delta}$ given by Lemma \ref{qQrOdN2Nmt} and the push-forward measure $\nu:= (\pi_{2}|_{J(\tilde{f})})_*(\tilde{\nu})$ under the restriction $\pi_{2}|_{J(\tilde{f})}$.
Then, we may define the restriction of $(J(\tilde{f}), \tilde{\mathcal{F}}, \tilde{\nu})$ to $\pi_1^{-1}(X_i(S))$ written as
$\left(J(\tilde{f}) \cap {\pi_1^{-1}(X_i(S))}, \tilde{\mathcal{F}}\cap {\pi_1^{-1}(X_i(S))}, \tilde \nu|_{\pi_1^{-1}(X_i(S))}\right)$.
Denote $\mathcal{F}_{i} := \tilde{\mathcal{F}}\cap {\pi_1^{-1}(X_i(S))}$. Also, we have $ J_{i}(S) = \pi_{2}\left(J(\tilde{f}) \cap {\pi_1^{-1}(X_i(S))} \right)$.
Then, the push-forward of $\mathcal{F}_{i}$ under $\pi_{2}|_{J(\tilde{f})}$ is given by
$$
\pi_{2}|_{J(\tilde{f})}(\mathcal{F}_{i}) := \ \left\{A \subset J_{i}(S) :  \pi_{2}|_{J(\tilde{f})}^{-1}(A) \in \mathcal{F}_{i} \right\}
$$
proves to be a $\sigma$-algebra on $J_{i}(S)$.
Hence, we may define a measure $\nu_i:= \left(\pi_{2}|_{J(\tilde{f})}\right)_*\left(\tilde \nu|_{\pi_1^{-1}(X_i(S))}\right)$ given by 
$$
\left(\pi_2|_{J(\tilde{f})}\right)_*\left(\tilde \nu|_{\pi_1^{-1}(X_i(S))}\right)(A) := \tilde{\nu}\left(\pi_{2}|_{J(\tilde{f})}^{-1}(A)\cap \pi_1^{-1}(X_i(S))\right)
$$
for each $\pi_{2}|_{J(\tilde{f})}(\mathcal{F}_{i})$-measurable set $A$.
 Then, we obtain
\begin{align*}
\nu_{i}(D(z, r) \cap J_{i}(S))
& = \tilde{\nu}\left(\pi_1^{-1}(X_i(S)) \cap \pi_{2}|_{J(\tilde{f})}^{-1}\left(D(z, r)\right)\right)\\
&\leq \tilde{\nu}\left(\pi_1^{-1}(X_i(S)) \cap \pi_{2}|_{J(\tilde{f})}^{-1}\left(D(z, K^{-1}r_{n})\right) \right) \notag \quad (\because (\ref{K4r63f7qhJ}))\\ 
& \leq \tilde{\nu}\left(\eta_{\xi|_{n}} (\tilde{f}^{n}(\boldsymbol{D}))\right) \quad (\because (\ref{inclusion_ita_from_below_new}))\notag\\
& = \int_{\tilde{f}^{n}(\boldsymbol{D})} \|(\gamma_{\xi|_{n}})^{\prime}\|^{\delta} d \tilde{\nu} \quad (\because \text{Lemma } \ref{3ni8tG7iYu})\notag\\
& < (2KK_{1}\varepsilon^{-1})^{\delta}r^{\delta}. \quad (\because (\ref{enqlClDijr})(\ref{bhhlyOt8hr})) \notag
\end{align*}
 Recalling that $z\in J_i(S)$ is arbitrarily taken, we see that combining this with  Proposition 2.2 in \cite{97_falconer_techniques} and the fact that  $\nu_{i}(J_i(S))=1$ implies $\mathcal{H}^{\delta}(J(S)) \geq \mathcal{H}^{\delta}(J_{i}(S)) > 0$.
The proof is complete.
\end{proof}

\begin{thm}\label{cI0oM7KYMj}
Let $S$ be a non-elementary, expanding, irreducible and aperiodic rational GDMS. 
Then we have $\mathcal{H}^\delta(J(S)) < \infty$ and  $\overline \dim_{B}(J(S))\le \delta$.
\end{thm}

\begin{proof}
Under the same setting of Theorem \ref{thm_dimension_julia_delta} without the BSC, we shall use the same notations in its proof. 
First, observe that for all $n\in \N$,
                    \begin{equation}\label{inclusion_ita_from_above}
                     \pi_{2}|_{J(\tilde{f})}^{-1}\left(D(z, Kr_{n})  \cap J(S)\right) \supset  \pi_{2}|_{J(\tilde{f})}^{-1}\left(\gamma_{\xi|_{n}}(D(z_{n}, 2^{-1}\varepsilon)  \cap J(S))\right) \supset  \boldsymbol{D}.
                    \end{equation}
Since $\spt(\nu) = J(S)$, we see that for each small  $a>0$ there exists $M(a)>0$ such that $\nu\left(D(\zeta, a)\cap J(S)\right) > M(a)$ for all $\zeta \in J(S)$.  
Let $r>0$ and $n \in \N$ as in (\ref{K4r63f7qhJ}).

Then we obtain
\begin{align*} 
\nu(D(z, Kr_{n})\cap J(S))=&\tilde\nu\left(\pi_{2}|_{J(\tilde{f})}^{-1}\left(D(z, Kr_n)\cap J(S)\right)\right)\\
\geq& \tilde\nu\left(\eta_{\xi|_{n}}\left(\tilde{f}^{n}(\boldsymbol{D})\right)\right)\quad (\because (\ref{inclusion_ita_from_above}))\\
=& \int_{\tilde{f}^{n}(\boldsymbol{D})} \|(\gamma_{\xi|_{n}})^{\prime}\|^{\delta} d \tilde{\nu}  \quad (\because \text{Lemma } \ref{3ni8tG7iYu})\notag\\
\geq&  K^{-\delta}\|(\gamma_{\xi|_{n}})^{\prime}(z_n)\|^\delta \nu\left(D(z_{n}, 2^{-1} \varepsilon) \cap J(S) \right) \quad (\because \text{Lemma } \ref{koebe1})\\
                    \geq&  (2K\varepsilon^{-1})^{-\delta} M(2^{-1} \varepsilon)r_n^\delta
=(2K^2\varepsilon^{-1})^{-\delta} M(2^{-1} \varepsilon)(Kr_{n})^\delta.
\end{align*}
                It follows that
    $$
\frac{\nu(D(z, Kr_{n}))}{(Kr_{n})^\delta} 
\geq  (2K^{2}\varepsilon^{-1})^{-\delta} M(2^{-1} \varepsilon)>0.
$$
    
By Proposition 2.2 in \cite{97_falconer_techniques}, we conclude that $\mathcal{H}^\delta(J(S)) < \infty$ and $\overline \dim_{B}(J(S))\le \delta$.
The assertion has been justified.
\end{proof}

From the two preceding theorems, the corollary below immediately follows.
\begin{cor}\label{pYqFyMRHfD}
Let $S$ be a non-elementary, expanding, irreducible and aperiodic rational GDMS satisfying the backward separating condition.
Then, we have $\dim_{H}J(S) = \delta$ and $0 < \mathcal{H}^{\delta}(J(S)) < \infty$.
\end{cor}

\section{Example}\label{8dEsBgjjIq}
In this section, we shall use some results of this paper to present an example of an upper estimate of the Hausdorff dimension of a Julia set in $\widehat{\C}$ which is generated by a rational GDMS.

To define a rational GDMS denoted by $S$, we first consider a set of vertices $V_{S} := \{1,2,3\}$ and a set of direct edges $E_{S} := (V_{S} \times V_{S}) \setminus \{(1,2)\}$.
Let $n \in \N$ with $n \geq 5$.
We define rational maps $$g_{1}(z) := z^{n},\quad g_{2}(z) := (z - 3i)^{n} + 3i, \quad g_{3}(z) := (z - 3)^{n} + 3.
                $$
                We define the  family $(\Gamma_{e})_{e \in E_{S}}$ by $\Gamma_{e} := \{g_{t(e)}\}$ for each $e \in E$.

A rational GDMS $S := (V_{S}, E_{S}, (\Gamma_{e})_{e \in E_{S}})$ has now been defined. 
We next define an incidence matrix $A$ associated with the directed graph $(V_{S}, E_{S})$ as follows.
$$
\forall \mathbf{e} \in V_{S} \times V_{S},\quad a_{\mathbf{e}} := 
\begin{cases}1 & \text{if } \mathbf{e} \in E_{S} \\ 0 & \text{if } \mathbf{e} = (1,2) 
\end{cases}
\quad\text{;}\quad 
A := (a_{e})_{e \in E_{S}} := \left(\begin{array}{lll}1 & 0& 1 \\ 1 & 1 & 1 \\ 1 & 1 & 1\end{array}\right).
$$
Let $(\Sigma_{A}, \hat\sigma)$ denote the subshift of finite type over $V_{S}$ associated with $A$, where $\hat\sigma\colon \Sigma_{A} \to \Sigma_{A}$ is the left shift.   
Note that the matrix $A$ is irreducible and aperiodic. Since each $\Gamma_e$ is a singleton, $(\Sigma_{A}, \hat\sigma)$ is topologically conjugate to $(X(S),\sigma)$. Hence, $S$ is irreducible and aperiodic. 
The largest eigenvalue of $A$ will be denoted by $\lambda_{A} := (3 + \sqrt{5})/2$.

Let $G_{S}$ be the  rational semigroup generated by the family $(g_{i})_{i \in V_{S}}$ and $F(G_{S})$ be the Fatou set of $G_{S}$, and the post-critical set of $G_{S}$ is defined by $P(G_{S}):=\overline{\bigcup_{h \in G_{S}}\{\text{all critical values of } h\}} \subset \widehat{\C}$.
We write $\CP$ to refer to the set of all the critical points of elements of $\{g_{i}: i \in V_{S}\}$.

Then, one may check that the system $S$ is hyperbolic as follows.
Let 
$$
K_{S} := \{|z| \geq 5\} \cup \bigcup_{i \in V_{S}}\{|z - p_{i}| \leq 2^{-1}\} \subset F(G_{S}).
$$
Then, it immediately follows that $g(\CP) \subset K_{S}$ for each $g \in \{g_{i} : i \in V_{S}\}$.
Hence, any small open neighborhood $g(\mathcal{N}_{S}) \subset \mathcal{N}_{S}$ for  each element of $\{g_{i} : i \in V_{S}\}$.
Then, Montel's theorem implies that $\mathcal{N}_{S} \subset F(G_{S})$, and thus $P(G_{S}) \subset F(G_{S})$.
Thus, we conclude that the system $S$ is hyperbolic.

Let $\tilde{f}_{S}$ be a skew product map associated with $S$ and write $\tilde{f} := \tilde{f}_{S}|_{J(\tilde{f}_{S})}$ for simplicity.
Now we see that Proposition \ref{Oqeiwu19LMc} implies that $S$ is expanding. 
Let $\delta(S)$ denote the unique zero of the topological pressure $P(\tilde{f}, t\tilde\varphi)$, where $\tilde\varphi$ and $P(\tilde{f}, t\tilde\varphi)$ are as in subsection \ref{qdXaxWWPnV}.
Also, by \ref{Lg4K3pcyFt} and \ref{cmu6vrJlgl} of Lemma \ref{qQrOdN2Nmt}, there exists an equilibrium state $\hat{\nu}$ of $(\tilde{f}, \delta(S)\tilde\varphi)$ and we obtain
$$
\delta(S) = \frac{h_{\hat{\nu}}(\tilde{f})}{\int \log |\tilde{f}^{\prime}|d\hat{\nu}}.
$$

By \cite[Theorem B, (ii)]{MR1785403}, we have  $h_{\hat{\nu}}(\tilde{f}) \leq h_{{\pi_{1}}_{*}\hat{\nu}}(\sigma) + \log n$. Hence, $h_{\hat{\nu}}(\tilde{f}) \leq h_{\text{top}}(\sigma)+ \log n$. 
Since the pairs $(X(S), \sigma)$ and $(\Sigma_{A}, \hat\sigma)$ are topologically conjugated by some map $\psi: \Sigma_{A}  \to X(S)$ given by $(\omega_{n})_{n \in \N} \mapsto (g_{\omega_{n + 1}}, (\omega_{n}, \omega_{n + 1}))_{n \in \N}$, we have $h_{\text{top}}(\sigma) = h_{\text{top}}(\hat{\sigma}) =\log\lambda_{A}$.

One may verify that the system $S$ satisfies the BSC as follows.
Let $U :=\{|z| \leq 5\}$.
Then, ${g_{1}}^{-1}(U) \subset \{|z| \leq \sqrt[5]{5}\}<\{|z|<1.4\}$, 
$g_2^{-1}(U) \subset\{|z-3i|<\sqrt[5]{8}\} \subset\{|z-3i|<1.55\}$ and $g_3^{-1}(U) \subset\{|z - 3| \leq \sqrt[5]{8}\}<\{|z - 3|<1.55\}$.
Thus, $g_i^{-1}(U) \subset U$, for each $i \in V_{S}$ and $g_i^{-1}(U) \cap g_j^{-1}(U)=\emptyset$ for $i \neq j$.
Since $J(S) \subset J(G_{S}) \subset U$, one may conclude that $S$ satisfies the BSC.

Therefore, Corollary \ref{pYqFyMRHfD} yields
\begin{align*}
\dim_{H}(J(S)) = \delta(S) = \frac{h_{\hat{\nu}}(\tilde{f})}{\int \log |\tilde{f}^{\prime}|d\hat{\nu}} \leq \frac{1}{\log n}\left(\log n + \log\frac{3 + \sqrt{5}}{2} \right) < 1 +  \frac{\log(3 + \sqrt{5})2^{-1}}{\log 5} < 1.6
\end{align*}
and $0 < \mathcal{H}^{\delta(S)}(J(S)) < \infty$.

\end{document}